\documentclass[11pt]{article}
\usepackage{amsmath,amsfonts,amsthm,amscd,amssymb,graphicx}
\usepackage{color}

\numberwithin{equation}{section}
\newtheorem{theorem}{Theorem}[section]

\newtheorem{lemma}[theorem]{Lemma}

\newtheorem{proposition}[theorem]{Proposition}

\newtheorem{corollary}[theorem]{Corollary}

\newtheorem{remark}[theorem]{Remark}

\def\cQ {\mathcal{Q}}
\def\bR  {\mathbb{R}}
\def\del{\partial }
\def \la {\langle}
\def \ra {\rangle}
\def \cE {\mathcal{E}}

\def\RR{\mathbb{R}}

\begin{document}

\title{The Maxwell-Boltzmann approximation for ions kinetic modeling}
\author{Claude Bardos\footnotemark[1]  \and Fran\c{c}ois Golse\footnotemark[2] \and Toan T. Nguyen\footnotemark[3] \and R\'emi Sentis \footnotemark[1] 
}

\maketitle

\renewcommand{\thefootnote}{\fnsymbol{footnote}}

\footnotetext[1]{Laboratoire J.-L. Lions, BP187, 75252 Paris Cedex 05, France. Emails: claude.bardos@gmail.com; sentis.remi@gmail.com}

\footnotetext[3]{Department of Mathematics, Pennsylvania State University, State College, PA 16802, USA. Email: nguyen@math.psu.edu. TN's research was supported in part by the NSF under grant DMS-1405728.
}

\footnotetext[2]{Ecole Polytechnique, Centre de Math\'ematiques Laurent Schwartz, 91128 Palaiseau Cedex,
France, Email: francois.golse@polytechnique.edu)
}


\begin{abstract}
This paper aims to justify the Maxwell-Boltzmann approximation for electrons, preserving the dynamics of ions at the kinetic level. Under sufficient regularity assumption, we provide a precise scaling where the Maxwell-Boltzmann approximation is obtained. In addition, we prove that the reduced ions problem is well-posed globally in time.  
\end{abstract}

\tableofcontents

\section{Introduction}


\subsection{Physical framework for the modeling}

Consider a plasma consisting of electrons and one kind of ions, which are
charged particles moving in an electromagnetic field. Let $\widetilde{f}%
_{+}(x,v,t)$ and $\widetilde{f}_{-}(x,w,t)$ be the corresponding density
distribution functions for ions and electrons, respectively; here, $(v,w)$
represent particle velocity variables for ions and electrons belonging to $%
\mathbb{R}^{d}$ (here $d=2$ or $3$),  and $x$ denotes the space variable
belonging to a periodic torus or an open set of $\mathbb{R}^{d}$ with a boundary, and $t$ is the time. In absence of magnetic fields, the
dynamics of the plasma is
modeled by the following well-known system 
\begin{eqnarray}
\partial _{t}\widetilde{f}_{-}+w\cdot \nabla _{x}\widetilde{f}_{-}-\frac{q_{e}%
}{m_{e}}\widetilde{E}\cdot \nabla _{w}\widetilde{f}_{-} &=&\widetilde{Q}_{-}(%
\widetilde{f}_{-})  \label{VB-m} \\
\partial _{t}\widetilde{f}_{+}+v\cdot \nabla _{x}\widetilde{f}_{+}+\frac{q_{e}%
}{m_{i}}\widetilde{E}\cdot \nabla _{v}\widetilde{f}_{+} &=&0
\end{eqnarray}%
where $m_{e},m_{i}$ denote the electrons and ions mass, $q_{e}$ the
elementary charge (for the sake of simplicity we assume that the ions charge
is equal to $1$). The electrostatic field is given by $\widetilde{E}=-\nabla
_{x}\widetilde{\phi }$ and solves the Poisson equation:
\begin{equation*}
- \varepsilon^0\Delta _{x}\widetilde{\phi }=\langle 
\widetilde{f}_{+}\rangle -\langle \widetilde{f}_{-}\rangle 
\end{equation*}
with $\varepsilon ^{0}$ being the vacuum permittivity. Here and in the sequel, $\left\langle \cdot\right\rangle $ denotes the
integral on the velocity space, that is $\la F \ra: = \int_{\mathbb{R}%
^d} F(v)dv. $ In equation \eqref{VB-m}, $\widetilde{Q}_{-}(%
\widetilde{f}_{-})$ accounts for the 
collisional operator of
electrons with themselves (for example, a binary Boltzmann or Fokker-Planck operator). We have assumed that there is no
collision between electrons and ions and of course no binary collision of
ions with themselves. For interaction between disparate masses between particles, see, for instance, \cite{DL96a,DL96b}.

\bigskip

Such a model has been widely used in plasma physics from a theoretical point
of view; see, for instance, \cite{GPbook, Nbook,SePlasma,TKbook}. But, since the
electron/ion mass ratio is small, the characteristic time scale of the
dynamics of ions is significantly larger than that of electrons. As a
consequence, if one addresses a model for the ions dynamics, it is very classical
to use a fluid modeling for the electrons, assuming they have reached the thermal equilibrium; that is to say, the distribution function is a
Maxwellian function with an electrons temperature $\widetilde{\theta }$ and a
density given by the well-known \emph{Maxwell-Boltzmann relation} 
\begin{equation}
\langle \widetilde{f}_{-} \rangle =e^{q_{e}\widetilde{\phi }/%
\widetilde{\theta }}  \label{MBe}
\end{equation}%
(the temperature $\widetilde{\theta }$ can be expressed in energy units).

\bigskip

\emph{In this paper, we aim to justify the Maxwell-Boltzmann approximation for electrons} \eqref{MBe} from the kinetic model \eqref{VB-m}. This approximation has been used in a number of works; for instance, see
\cite{Bouchut, HK1,HK2}, among many others. Other important scalings involving the massless electrons limit (\cite{CDL, Degond,Grenier,Herda}), quasi-neutral approximations (\cite{HK1,HK3}), or large magnetic fields (\cite{Claudia}) may be compared with the present paper. We note in particular the work \cite{Guo} where the local Maxwellian for electrons is recovered, and instead of the Maxwell-Boltzmann relation, the isentropic relation $\langle \widetilde{f}_{-} \rangle \sim \tilde \theta^{3/2}$ is used.



\subsection{The non-dimensional form}

We denote by $\theta _{ref}$ and $N_{ref}$ the characteristic values of the
electrons temperature and of the electrons density, and introduce the
non-dimensional parameter 
\begin{equation*}
\varepsilon =\sqrt{\frac{m_{e}}{m_{i}}}
\end{equation*}%
assumed to be sufficiently small. 
To derive non-dimensional equations, let us rescale the velocity of electrons and their distribution
function as follows:  
\begin{equation*}
w=v/\varepsilon ,\qquad f_{-}(v)=\frac{1}{\epsilon^3 N_{ref}}\widetilde{f}_{-}(v/\varepsilon ),\qquad f_{+}(v)=\frac{1}{N_{ref}}\widetilde{f}_{+}(v).
\end{equation*}
Observe that the scaling preserves the local density $\int \widetilde{f}_{-}(w)dw=\int f_{-}(v)dv.$ We also introduce $\lambda _{D}$, the Debye length (e.g., see \cite{TKbook}), 
$$\lambda _{D}=\sqrt{\epsilon ^{0}\theta _{ref}/(q_{e}^{2}N_{ref})} $$
and set $\phi =q_{e}\widetilde{\phi }/\theta _{ref}$ and $\theta =\widetilde{%
\theta }/\theta _{ref}$. The scaled collisional operator, instead of $\widetilde{Q}_{-}(\widetilde{f}_{-})$, now reads \begin{equation*}
\eta _{\varepsilon }Q(f_{-})
\end{equation*}
for $\eta_\epsilon$ being a scaling parameter; the higher $\eta _{\epsilon }$, the more
collisional is the electrons population. 

\bigskip

In the sequel, we assume that the
plasma is collisional enough; precisely, we assume 
\begin{equation}
\qquad \lim_{\varepsilon \rightarrow 0}\eta _{\varepsilon }\varepsilon
^{-1}=\infty ,\qquad \lim_{\varepsilon \rightarrow 0}\eta _{\varepsilon
} < +\infty.  \label{scaling}
\end{equation}%
Using the above notations, the dynamics of $f_{-}$ and of $f_{+}$ then reads as
follows: 
\begin{eqnarray}
\varepsilon \partial _{t}f_{-}+v\cdot \nabla _{x}f_{-}+\nabla _{x}\phi \cdot
\nabla _{v}f_{-} &=&\eta _{\epsilon }Q(f_{-})  \label{Vlasov} \\
\partial _{t}f_{+}+v\cdot \nabla _{x}f_{+}-\nabla _{x}\phi \cdot \nabla
_{v}f_{+} &=&0  \label{Vlasoi}
\end{eqnarray}
and the Poisson equation for the electric potential $\phi $ reads as 
\begin{equation}
-\lambda _{D}^{2}\Delta _{x}\phi =\left\langle f_{+}\right\rangle
-\left\langle f_{-}\right\rangle .  \label{Poisson}
\end{equation}%

\bigskip

The spatial domain $\Omega $ will be a periodic torus or a bounded open subset of $\mathbb{R}%
^{d} $ with a boundary $\partial \Omega $. In the latter case, we
assume that both ions and electrons reflect specularly:
\begin{equation}
f_{\pm }(x,v,t)=f_{\pm }(x,v-2(v\cdot n(x))n(x),t),\qquad n(x)\cdot v<0\,
\label{bdry-specular}
\end{equation}%
at each point $x\in \partial \Omega $, in which $n(x)$ denotes the outward
normal vector of $\partial \Omega $. We also assume the Neumann boundary
condition for \eqref{Poisson} 
\begin{equation*}
\frac{\partial \phi }{\partial n}_{|_{\partial \Omega }}=0.
\end{equation*}
As for the initial conditions $f_{-}(0)$ and $f_{+}(0),$ in accordance with the
Neumann boundary condition of equation \eqref{Poisson}, we assume
\begin{equation*}
\int \la f_{+}(0)\ra dx=\int \la f_{-}(0)\ra dx=m_{0}
\end{equation*}
Finally, we assume that for each continuous and rapidly decaying function $f (v)$, the collisional operator $Q (\cdot)$ satisfies the following classical properties:

\begin{equation}
\la Q(f)\ra=0,\qquad \la vQ(f)\ra=0,\qquad \la|v|^{2}Q(f)\ra=0,
\label{Q-property}
\end{equation}
and the H-theorem
\begin{equation}
\la Q(f)\log f \ra \le 0,  \label{htheorem}
\end{equation}
with equality implying that such functions are local Maxwellians.

%

\subsection{Conservation properties}

We assume that $f_{-}$ and $f_{+}$ have sufficient regularity and rapidly
decay to zero as $v\rightarrow \infty $. The first property of $Q$ in %
\eqref{Q-property} immediately yields the conservation of mass: 
\begin{equation}
\partial _{t}\la f_{+}\ra+\nabla _{x}\cdot \la vf_{+}\ra=0,\qquad \partial
_{t}\la f_{-}\ra+\varepsilon ^{-1}\nabla _{x}\cdot \la vf_{-}\ra=0.
\label{massev}
\end{equation}%
Together with the specular reflection boundary condition for $f_{\pm }$,
this yields the global conservation of mass: 
\begin{equation}
\int \la f_{+}(t)\ra dx=\int \la f_{-}(t)\ra dx=m_{0},\qquad \forall t\geq 0.
\label{def-m0}
\end{equation}
For the momentum conservation, we get
\begin{equation}
\begin{aligned}
\partial _{t}\la vf_{+}\ra+\nabla _{x}\cdot \la v\otimes vf_{+}\ra&=-\nabla
_{x}\phi \cdot\la f_{+}\ra,  \label{uion}
\\
\partial _{t}\la vf_{-}\ra+\frac{1}{\varepsilon }\nabla _{x}\cdot \la %
v\otimes vf_{-}\ra&=\frac{1}{\varepsilon }\nabla _{x}\phi \cdot\la f_{-}\ra.
\end{aligned}
\end{equation}
Moreover, for the ions and electrons energy conservation, we get 
\begin{equation}
\begin{aligned}
\partial _{t}\la\frac{|v|^{2}}{2}f_{+}\ra+\nabla _{x}\cdot \la v\frac{|v|^{2}%
}{2}f_{+}\ra &=-\nabla _{x}\phi \cdot\la vf_{+}\ra  \label{enion}
\\
\partial _{t}\la\frac{|v|^{2}}{2}f_{-}\ra+\frac{1}{\varepsilon }\nabla
_{x}\cdot \la v\frac{|v|^{2}}{2}f_{-}\ra &=\frac{1}{\varepsilon }\nabla
_{x}\phi \cdot\la vf_{-}\ra . 
\end{aligned}\end{equation}
Hence, a direct computation yields 
\begin{eqnarray*}
\frac{d}{dt}\int \la\tfrac{1}{2}|v|^{2}f_{+}\ra+\la\tfrac{1}{2}|v|^{2}f_{-}%
\ra dx &=&\int_{\Omega }\nabla _{x}\phi \cdot \Big (-\la vf_{+}\ra+\frac{1}{%
\epsilon }\la vf_{-}\ra\Big )dx \\
&=&\int_{\Omega }\phi \nabla _{x}\cdot \Big(\la vf_{+}\ra-\frac{1}{\epsilon }%
\la vf_{-}\ra\Big)dx \\
&=&\int_{\Omega }\phi \partial _{t}\Big(\la f_{-}\ra-\la f_{+}\ra\Big)dx
\end{eqnarray*}%
in which the conservation \eqref{massev} of mass was used. Using the Poisson
equation \eqref{Poisson} and the integration by parts $\int \phi \Delta
(\partial _{t}\phi )dx=-\int \nabla \phi \cdot (\partial _{t}\nabla \phi )dx$ into
the above computation, we obtain the conservation of energy
\begin{equation}
\int \la\frac{|v|^{2}}{2}f_{-}\ra+\la\frac{|v|^{2}}{2}f_{+}\ra dx+\frac{%
\lambda _{D}^{2}}{2}\int_{\Omega }|\nabla _{x}\phi |^{2}\;dx=\cE_{0},\qquad
\forall t\geq 0  \label{def-E0}
\end{equation}
with $\cE_{0}$ being a
constant. Finally, multiplying equation \eqref{Vlasov} by $\log f_{-}$, we obtain 
\begin{equation}
\frac{d}{dt}\int \la f_{-}\log f_{-}\ra dx+\frac{1}{\varepsilon }\eta
_{\varepsilon }\int \la Q(f_{-})\log f_{-}\ra dx  \label{id-entropy} =0.
\end{equation}%
In particular, by \eqref{htheorem}, the entropy of $f_{-}$ is decreasing in time: 
\begin{equation*}
\frac{d}{dt}\int \la f_{-}\log f_{-}\ra dx\leq 0.
\end{equation*}


\subsection{Formal Maxwell-Boltzmann approximation}

{In this section the word {\it formal} refers to the fact that the propositions below are proven under some extra regularity assumption which is reasonable but may not be easy to establish under the present knowledge of the subject.}

 Let $m_0, \cE_0$ be the constants defined as in \eqref{def-m0} and \eqref{def-E0}.
Again, we assume that $f_-$ and $f_+$ have sufficient regularity and rapidly decay to zero as $v\to \infty$. Assume that $\Omega$ is non-axisymmetric. We have the following formal result.

 \begin{proposition}\label{theo-formal} 
 	Assume \eqref{scaling}. Let $(f^\epsilon_-, f^\epsilon_+, \phi^\epsilon)$ be a smooth solution to system \eqref{Vlasov}- \eqref{Poisson} so that 
	$$ |f_-^\epsilon (x,v,t)| \le C e^{-|v|^\gamma} $$
	for some positive constants $C, \gamma$, uniformly in $x,v,t$ and in $\epsilon$.  	
	Then, on any finite time interval $[0,T]$, $\la f^\epsilon_- \log f^\epsilon_- \ra $ is uniformly bounded in $L^1_x$. Assume further that as $\epsilon \to 0$, the functions  
   $ (f^\epsilon_+,f^\epsilon_-, \phi^\epsilon)$ converge in a weak sense.
   Then, the limit  $ (\overline {f_+}, \overline {f_-},  \phi)$ must satisfy 
\begin{equation}\label{MBBB}
 \overline{f_-}(x,v,t)=n_e (x,t)\Big (\frac{\beta(t) }{2\pi}\Big)^{\frac{d}2} e^{-\beta(t) \frac{ |v|^2}2}, \qquad n_e(x,t) = e^{\beta(t) \phi(x,t)}
 \end{equation}
where $f_+ (x,v,t) ,\phi (x,t) , \beta (t)$ solve the following system  
 \begin{equation}
 \begin{aligned}
 & 
 \partial_t\overline{f_+ }+ v \cdot \nabla_x \overline{f_+ } - \nabla_x \phi \cdot \nabla_v \overline{f_+ }= 0,
\\ &
-\lambda _{D}^{2}\Delta \phi+ e^{\beta(t) \phi}= \la \overline{f_+}\ra, \\
 &   \frac{ m_0 d}{2 \beta(t)} +\int_{\Omega} \la \frac{|v|^2}2  \overline{f_+}\ra  dx + \frac 1 2\int_\Omega  |\nabla_x \phi(x,t)|^2 dx = \cE_0 . \end{aligned}
 \end{equation}
 \end{proposition}

\begin{remark}{\em The relaxation to the equilibrium of the form of a Maxwellian as in \eqref{MBBB} is precisely due to the presence of the collision operators, without which the equilibrium is of the form 
 $$ \overline{f_-}(x,v,t) = \mu( \frac{|v|^2}{2} - \phi)$$
for any function $\mu(\cdot)$,  with $\phi$ solving the Poisson equation 
$$-\lambda _{D}^{2}\Delta \phi+ \int_{\RR^d} \mu( \frac{|v|^2}{2} - \phi) \; dv = \la \overline{f_+}\ra.$$
}\end{remark}

\begin{remark}{\em We note that there is no time-dynamics for the electrons in the limit of $\epsilon \to 0$. The time-dependence is precisely through the dynamics of ions. 
If we denote $$ n_I (x,t) = \la \overline{f_+} (x,\cdot,t) \ra, $$
the Poisson equation now reads
\begin{equation}\label{MB}-\lambda _{D}^{2}\Delta \phi + e^{\beta \phi}= n_I\end{equation}
and is often referred to as the Poisson-Poincare equation. 
}\end{remark}

 We now consider the following system with a collisional operator for ions
  \begin{eqnarray}
 \varepsilon \partial _{t}f_{-}+v\cdot \nabla _{x}f_{-}+\nabla _{x}\phi \cdot
 \nabla _{v}f_{-} &=&\eta _{\epsilon }Q(f_{-}) 
  \label{Vlaso2} \\
 \partial _{t}f_{+}+v\cdot \nabla _{x}f_{+}-\nabla _{x}\phi \cdot \nabla
 _{v}f_{+} &=& \sigma_\epsilon Q_+(f_{+})  
  \label{Vlaso22}
 \end{eqnarray}
 coupled with \eqref{Poisson}. Our second formal result is as follows. 
 
 \begin{proposition}\label{the2}
 Assume \eqref{scaling} and that $ \lim_{\epsilon \to 0} \sigma_\epsilon= \infty$. Let $(f^\epsilon_-, f^\epsilon_+, \phi^\epsilon)$ be a smooth solution to system \eqref{Vlaso2}, \eqref{Vlaso22}, and \eqref{Poisson},
so that 
	$$ |f_\pm^\epsilon (x,v,t)| \le C e^{-|v|^\gamma} $$
	for some positive constants $C, \gamma$, uniformly in $x,v,t$ and in $\epsilon$.  	 
 Assume that as $\epsilon \to 0$, the functions  
 $ (f^\epsilon_+,f^\epsilon_-, \phi^\epsilon)$ converge in a weak sense.
 Then, the limit  $ (\overline {f_+}, \overline {f_-},  \phi)$ are local Maxwellians of the form 
 
 \begin{equation}
 \begin{aligned}
  \overline{f_+}(x,v,t) &=n_I(x,t)\Big (\frac1{2\pi \theta_I}\Big)^{\frac{d}2} e^{-\frac{ |v-u_I|^2}{2 \theta_I} },
  \\
 \overline{f_-}(x,v,t)&=n_e(x,t)\Big (\frac{\beta }{2\pi}\Big)^{\frac{d}2} e^{-\beta \frac{ |v|^2}2}, \qquad n_e(x,t) = e^{\beta \phi(x)}
\end{aligned} \end {equation}
in which $(n_I(x,t), u_I(x,t),\theta_I(x,t))$ and  $(\beta (t) ,\phi (x,t) )$ solve the following compressible Euler-Poisson system
 \begin{equation}
 \begin{aligned}\label{Euler-Poisson}
 &\del_t n_I + \nabla  \cdot (n_I u_I)=0,\\
 &\del_t (n_I u_I)+ \nabla\cdot ( n_I u_I\otimes u_I ) +  \nabla (n_I \theta_I) + n_I \nabla \phi =0, \\
  & \del_t \Big (n_I(\frac{|u_I|^2}2 + \frac{d}{2}\theta_I)\Big) + \nabla\cdot \Big( n_I u_I \Big (\frac{|u_I|^2}2 + \frac{d+2}{2}\theta_I\Big) \Big) +
  n_I u_I \cdot \nabla \phi=0,\\
  &-\lambda _{D}^{2}\Delta \phi+ e^{\beta \phi}= n_I ,\\
  &  \frac{ m_0 d}{2 \beta} +\int_{\Omega} n_I(x,t)\Big (\frac{|u_I|^2}2 + \frac d 2 \theta_I\Big) \; dx+ \frac {\lambda _{D}^{2}} 2\int_\Omega |\nabla \phi(x,t)|^2\; dx = \cE_0  .
  \end{aligned}
 \end{equation}
  \end{proposition}

For the proofs, we shall use the following lemma (cf. \cite{DV02} or \cite[Proposition 13]{DV05} for discussions on more general setting). 

\begin{lemma}[Korn's inequality]\label{lem-Korn} Let $\Omega $ be a smooth bounded subset of $\mathbb{R}^d$, $d\ge 2$. Then, there exists a constant $\overline K (\Omega)>0$ such that for any  vector fields $u:Ê\Omega \mapsto \mathbb{R}^d$, one has
 \begin{equation} \label{korn}
 \Big \| \frac{\nabla u + \nabla u^t}2\Big \|_{L^2(\Omega)} \ge \overline K(\Omega)  \inf_{R\in \mathcal R (\Omega)} \|\nabla(u-R)\|^2_{L^2(\Omega)} ,
 \end{equation}
in which $\mathcal R(\Omega)$ denotes the space that consists of all affine maps $R: \Omega \mapsto \mathbb{R}^d$ whose linear part is anti-symmetric.  In particular, if $\Omega$ is non-axisymmetric and if $u\cdot n = 0$ on $\partial\Omega$, then the Korn's inequality \eqref{korn} holds for $R\equiv 0$.  
\end{lemma}

 \begin{proof}[Proof of proposition \ref{theo-formal}] 
 We first prove that $\overline{f_-}$ is of the form of a local Maxwellian. Indeed, by a view of \eqref{id-entropy}, together with the assumption $\lim_{\epsilon \to0} \eta_\epsilon \epsilon^{-1} = \infty$, we obtain in the limit 
$$
 \int_0^T \iint_{\Omega\times \bR^d} \cQ(\overline{f_-}) \log \overline {f_-} \; dvdxdt  = 0.$$
By the H-theorem, $\overline{f_-}$ is a local Maxwellian of the form 
$$ \overline{f_-}(x,v,t)=n_e \Big (\frac{\beta}{2\pi}\Big)^{\frac{d}2} e^{-\beta\frac{ |v - u_-|^2}2}
$$
in which $(n_e, u_-, \beta)$ depend on $(x,t)$. In particular, $\cQ(\overline{f_-}) = 0$. By a view of \eqref{scaling}, the Vlasov-Boltzmann equation for $\overline{f_-}$ in the limit of $\epsilon \to 0$ becomes 
\begin{equation}\label{Vf-} v \cdot \nabla_x \overline{f_-}  + \nabla_x \cdot \nabla_v \overline{f_-} = 0, \qquad \forall ~ (x,v) \in \Omega \times \bR^d.\end{equation}
Direct computations yield 
$$
v \cdot \nabla_x \overline{f_-} =  v\cdot \Big[ \nabla \log n_e - \frac d2\nabla\beta  + \frac{\beta|v-u_-|^2}{2}\nabla \beta +\beta  \sum_k (v_k - u_{k,-})\nabla u_{k,-} \Big] 
\overline{f_-} $$
and 
$$  \nabla_x \phi \cdot \nabla_v \overline{f_-} = -  \beta \nabla_x \phi  \cdot (v-u_-) \overline{f_-}.$$
We write \eqref{Vf-} as a polynomial with variable $v-u$, and set its coefficients to be zero. From the cubic term, we get $\nabla \beta =0$ and so $\beta = \beta(t)$.  The quadratic term is 
$$ \overline{f_-} \beta[(v-u_-)\otimes (v-u_-)]: \frac{\nabla u _-+ \nabla u_-^t}{2}  = \overline{f_-} \beta \sum_{jk} (v_j - u_{j,-}) (v_k - u_{k,-}) \frac{\partial_{x_j} u_{k,-} + \partial_{x_k} u_{j,-}}{2}$$
which implies that $\nabla u_- + \nabla u_-^t =0$. In addition, since $\overline{f_-}$ is an even function with respect to variable $v-u_-$, we get 
\begin{equation}\label{def-u} u_-(x,t) = \frac{1}{n_e(x,t)}\la  v \overline{f_-}(x,v,t) \ra .\end{equation}
This gives $u_-\cdot n =0$ on $\partial\Omega$, thanks to the specular boundary condition on $\overline{f_-}$. By Korn's inequality, $\nabla u_- =0$ and so $u_-=0$. The equation \eqref{Vf-} simply reduces to 
$$
\begin{aligned}
0 &=  \nabla \log n_e - \beta \nabla_x \phi. 
\end{aligned}$$
This proves that $n_e(x,t) = e^{\beta (t)\phi(x,t)}$ and $\overline{f_-}(x,v,t)$ is of the form as claimed. This completes the proof.
 \end{proof}

 \begin{proof}[Proof of proposition \ref{the2}] 
The proof is similar, yielding the same Maxwellian for $\overline{f_-}$. In addition, the assumption $ \lim_{\epsilon \to 0} \sigma_\epsilon= \infty$ implies that  $\overline{f_+}$ is  also a  local Maxwellian, as claimed. The macroscopic equations (\ref{Euler-Poisson}) are obtained by taking the moments of $\overline{f_+}$, upon recalling that 
$$ 
n_I = \la f_+ \ra, \qquad n_I u_I = \la v f_+ \ra, \qquad  n_I(\frac{|u_I|^2}2 + \frac{d}{2}\theta_I) = \la \frac{|v|^2}{2} f\ra .
$$
Indeed, same relations hold for $f_+^\epsilon$. By multiplying the Vlasov-Boltzmann equation for $f_+^\epsilon$ by $1, v$ and $\frac{|v|^2}{2}$ and integrating over $\mathbb{R}^d$ with respect to $v$, we obtain the following local conservation laws, respectively
$$\partial_t n_I^\epsilon + \nabla_x \cdot (n_I^\epsilon u_I^\epsilon)  =0,$$
$$\partial_t(n_I^\epsilon u_I^\epsilon)+ \nabla_x \cdot \la v \otimes v f_+^\epsilon \ra + n_I^\epsilon \nabla_x \phi  
=0,$$
$$\partial_t \Big[ n^\epsilon_I(\frac{|u^\epsilon_+|^2}2 + \frac{d}{2}\theta^\epsilon_+)\Big]  + \nabla_x\cdot \la v \frac{|v|^2}{2} f_+^\epsilon \ra  +  n_I^\epsilon u_I^\epsilon \cdot \nabla_x \phi=0.$$
Passing to the limit of $\epsilon \to 0$ and using the fact that the limiting distribution $\overline{f_+}$ is the Maxwellian (which is an even function in $v - u_I$), we compute 
$$
\begin{aligned}
  \nabla_x \cdot \la v \otimes v \overline{f_+} \ra & = \nabla_x \cdot \la u_I \otimes u_I \overline{f_+} \ra +  \nabla_x \cdot \la (v-u_I) \otimes (v-u_I) \overline{f_+} \ra
\\&= \nabla_x \cdot (n_I u_I \otimes u_I )+   \nabla_x (n_I \theta_I).
\end{aligned}   $$
Similarly, repeatedly using the evenness of $\overline{f_+}$ in $v-u_I$, we compute 
$$\begin{aligned}
\nabla_x \cdot \la v \frac{|v|^2}{2} \overline{f_+} \ra 
&= \nabla_x \cdot   \la u_I\frac{|v|^2}{2} \overline{f_+} \ra + \nabla_x \cdot  \la (v-u_I) \Big[ \frac{|v-u_I|^2}{2} + u_I \cdot v - \frac{|u_I|^2}{2}\Big] \overline{f_+} \ra
\\
&= \nabla_x \cdot  \Big( n_I u_I (\frac{|u_I|^2}2 + \frac{d}{2}\theta_I) \Big)+ \nabla_x \cdot  \la (v-u_I) u_I \cdot (v-u_I) \overline{f_+} \ra
\\
&= \nabla_x \cdot  \Big( n_I u_I (\frac{|u_I|^2}2 + \frac{d}{2}\theta_I) \Big)+ \frac{2}{d}\nabla_x \cdot  \la u_I \frac{|v-u_I|^2}{2} \overline{f_+} \ra
\\
&= \nabla_x \cdot  \Big( n_I u_I (\frac{|u_I|^2}2 + \frac{d}{2}\theta_I) \Big)+ \nabla_x \cdot  (n_Iu_I\theta_I ).
  \end{aligned}$$
This yields \eqref{Euler-Poisson}, and thus completes the proof of the theorem. 
 \end{proof}
\begin{remark}
{\em Letting  $\lambda_D\rightarrow 0$ in (\ref{Poisson}) or in its avatars (\ref{MB}) and (\ref{Euler-Poisson}) corresponds to the so called quasi-neutral approximation and leads formally to the relation 
\begin{equation}
\beta\nabla \phi \simeq\nabla (\log n_I) \label{quasineutre}\,.
\end{equation}
From (\ref{quasineutre}), one may deduce the formula
\begin{equation}
n_I \nabla \phi\simeq \nabla ( n_I \beta ^{-1} )\label{electrons}
\end{equation}
which means that the gradient of potential is the gradient of the electrons pressure.
The approximations (\ref{quasineutre}) and (\ref{electrons}) are well established at the level of physics (cf. \cite{SePlasma}). On the other hand the mathematical (with full rigor) justification of (\ref{quasineutre}) is the object of many recent works (cf. for instance \cite{HK1,HKM,HKN,HK3} and the references therein).
}\end{remark}


 
%

\section{Analysis of electrons system when the ions density is frozen} In
 this section,  the ions density $n_I (x)$ and  the kinetic energy of ions are taken independent of the time. For sake of presentation, we take the Debye length $\lambda_D$ equal to $1$.  
 
\subsection{Determination of the electrons temperature}
 In view of the formal derivation in the previous section with the time dependence only through the dynamics of ions, we study the stationary equation for electrons (denoting the electrons density distribution $f_- = f_- (x,v)$): 
 \begin{equation}\label{E-MB}
  \begin{aligned}
  v\cdot \nabla_x f_- + \nabla_x \phi \cdot\nabla_v f_- &=\eta \cQ (f_-) , \qquad \eta>0
  \\
  -\Delta \phi + \la f_-(x, \cdot)\ra &= n_I (x), \qquad  \frac{\partial \phi}{\partial n}_{\vert_{\partial \Omega}} = 0
  \end{aligned}\end{equation}
together with the specular boundary condition for $f_-$ on $\partial \Omega$, and the mass and energy constraints 
\begin{equation}\label{def-ME} 
\begin{aligned}
\int_\Omega  \la  f_-(x)\ra dx = \int_\Omega n_I (x) \; dx &=& m_0
\\
\int_\Omega  \la \frac{|v|^2}{2}  f_-(x)\ra dx + \frac12 \int_\Omega |\nabla \phi|^2 \; dx &=& \cE_1
\end{aligned}\end{equation}
for some fixed positive $\cE_1 =\cE_0 - \int  \la \frac{|v|^2}{2}  f_+ (x)\ra dx$. 

\begin{theorem}\label{theo-MB}
Let $\Omega$ be smooth, bounded, and non-axisymmetric, and let $f_-(x,v) $ be a solution to the Vlasov-Boltzmann equation \eqref{E-MB}. Assume that $f_-$ is continuous and rapidly decaying, and $-\log f_-$ has polynomial growth in $v$, as $v\to \infty$. 
Then $ f_-$  is given by the formula:
\begin{equation}\label{def-M}
f_- (x,v) = \Big(\frac{\beta}{2\pi}\Big )^{d/2} e^{-\beta (\frac{|v|^2}2 -\phi(x))} 
\end{equation}
with $\beta>0$ being $x-$independent and $\phi$ solution of the following elliptic problem
\begin{equation}
-\Delta \phi + e^{\beta \phi} = n_I (x), \qquad \, \quad \frac{\partial \phi }{\partial n}_{\vert_{\partial \Omega}}=0 . \label{Maxbol1}
\end{equation} 
\end{theorem}
\begin{proof} The proof is identical to that of Theorem \ref{theo-formal} in deriving the form of Maxwellian for electrons. 
\end{proof}

 With $f_- $ being the Maxwellian defined as in \eqref{def-M}, a direct computation yields 
$$ \la f_- (x) \ra  = e^{\beta \phi(x)}, \qquad 
\int \la \frac{|v|^2}{2} f_- \ra dx = \frac{m_0d}{2 \beta}.$$

\begin{remark}\label{rem-torus}
{\em In the case when $\Omega$ is axisymmetric, nonzero macroscopic velocity is allowed. For instance, when $\Omega = Q \times \mathbb{T}^k$  with $Q\subset \mathbb{R}^{d-k}$ being non axisymmetric, the failure of the Korn's inequality yields the following from of Maxwellian for $f(x,v)$ 
$$
f_- (x,v) = \Big(\frac{\beta}{2\pi}\Big )^{d/2} e^{-\beta \frac{|v - u|^2}2 } e^{\beta \phi(x)} 
$$
in which $u = (0, u_k)$ is a vector constant in $\mathbb{R}^{d-k}\times \mathbb{R}^k$. Necessarily, $\phi(x)$ is constant along the velocity field $u$. That is, $u \cdot \nabla \phi =0$.}
\end{remark}



\begin{remark}\label{rem-rtorus}
{\em 
Consider $\Omega$ to be a solid torus, defined by 
\begin{equation}\label{def-solidtorus}\Omega = \Big\{x = (x_1,x_2,x_3)\in \mathbb{R}^3~:~ \Big(a - \sqrt{x_1^2+x_2^2}\Big)^2 + x_3^2 < 1 \Big\}, \qquad a>1,\end{equation}
which can be parametrized with the following toroidal coordinates $(r,\theta,\varphi)$:
$$\begin{aligned}
 x_1 = (a+ r\cos \theta ) \cos \varphi , \quad x_2 = (a+r \cos \theta) \sin \varphi, \quad x_3=  r \sin \theta.\end{aligned}$$
Here, $0\le r\le 1$ is the radial coordinate in the minor cross-section, $0\le \theta<2\pi$ is the poloidal angle, 
and $0\le \varphi <2\pi$ is the toroidal angle. Let  $e_\varphi$ be the toroidal direction with respect to the angle $\varphi$. Then, 
the Maxwellian of $f(x,v)$ is of the form 
$$
f_- (x,v) = \Big( \frac{\beta}{ 2\pi } \Big)^{3/2}  e^{- \beta\frac{ |v - u_\varphi e_\varphi |^2}{2}} e^{\beta \phi (x)}, 
$$
for $u_\varphi = \gamma_\varphi (a+r\cos \theta)$, with $\gamma_\varphi$ being a constant, which can be determined from the conservation of angular momentum along the toroidal direction; see \cite{NStr}. 
}
\end{remark}


%

To determine $\beta$, we prove the following theorem.

\begin{theorem}\label{theo-beta} Let $\Omega$ be a bounded domain and $\cE_1 >0$. Fix a nonnegative ion density $n_I (x) \in L^2(\Omega)$ with finite mass $m_0$. Then, there exists a unique solution $(\beta, \phi)$ to the following elliptic problem:
 \begin{equation}\label{uni-MB}
-\Delta \phi + e^{\beta \phi} = n_I (x), \qquad \, \quad \frac{\partial \phi }{\partial n}_{\vert_{\partial \Omega}}=0 
\end{equation} 
together with the mass and energy constraints 
\begin{equation}\label{uni-ME} 
\begin{aligned}
\int_{\Omega} e^{\beta \phi} \; dx & = \int_\Omega n_I (x) \; dx = m_0,
\\
\cE(\beta): = \frac{m_0 d}{2\beta} &+ \frac12 \int_\Omega |\nabla \phi|^2 \; dx = \cE_1 .
\end{aligned}\end{equation}
\end{theorem}

\begin{proof} For each fixed $\beta>0$, the mapping $\phi \mapsto e^{\beta \phi}$ is strictly increasing and hence by the standard elliptic theory, the problem \eqref{uni-MB} has a unique solution $\phi^\beta \in H^2(\Omega)$. Next, to study the $\beta$-dependence, we consider the following linear problem for $\del_\beta \phi^\beta$:
\begin{equation}\label{delb}-\Delta \del_\beta \phi^\beta +\beta e^{\beta\phi^\beta} \del_\beta \phi^\beta =-e^{\beta\phi^\beta} \phi^\beta \,,\qquad \frac{\partial \del_\beta \phi^\beta}{\partial n}_{\vert_{\partial\Omega}} =0
\end{equation}
whose solution exists and is unique, with $\partial_\beta \phi^\beta \in H^2(\Omega)$. The uniqueness proves that $\del_\beta \phi^\beta$ is indeed the derivative of $\phi^\beta$ with respect to $\beta$.  

Next, to determine $\beta$, we use the energy constraint. Taking the $\beta$-derivative of the energy, we have 
\begin{equation}
\del_\beta  \cE(\beta) = -\frac{ m_0 d} {2 \beta^2}+ \int_\Omega \nabla_x \phi^\beta\cdot \nabla_x  \del_\beta \phi^\beta dx= -\frac{ m_0 d} {2 \beta^2} -\int_\Omega \phi^\beta\Delta_x  \del_\beta \phi^\beta dx\,. \label{delen}
\end{equation}
To compute the last term, from \eqref{delb}, we write 
$$
  \phi^\beta =   e^{-\beta\phi^\beta}( \Delta_x \del_\beta \phi^\beta) - \beta \del_\beta \phi_\beta
$$
which yields at once
$$
\begin{aligned}
&\del_\beta  \cE(\beta) =  -\frac{ m_0 d} {2 \beta^2} -  \int_\Omega e^{-\beta \phi^\beta}|\Delta_x  \del_\beta \phi^\beta|^2 dx  - \beta \int_\Omega |\nabla_x\del_\beta\phi^\beta |^2dx .
\end{aligned}
$$
This proves that $\beta \mapsto \cE(\beta )$ is a strictly decreasing function. Clearly, $\lim_{\beta\rightarrow 0}\cE(\beta)  =\infty$, which follows from the term $\frac{ m_0 d} { 2\beta}$. On the other hand, from the elliptic equation for $\phi^\beta$, we obtain 
$$
\begin{aligned}
\int_\Omega  |\nabla \phi^\beta|^2 dx &=  \int_\Omega \Big( n_I (x) \phi^\beta(x) -  e^{\beta \phi^\beta}\phi^\beta \Big)dx 
\\
&\le  \int_{\{\phi^\beta \ge 0\}} (n_I (x) \phi^\beta(x)-    e^{\beta \phi^\beta}\phi^\beta )dx  - \frac 1{\beta} \int_{\{\phi^\beta\le 0\}} e^{\beta \phi^\beta}Ê\beta \phi^\beta dx.
 \end{aligned}
$$
Using the fact that $e^x \ge x$ for $x\ge 0$ and $-x e^x \le e^{-1}$ for $x\le 0$, we obtain  
$$ \int_{\{\phi^\beta \ge 0\}} (n_I (x) \phi^\beta(x)-    e^{\beta \phi^\beta}\phi^\beta )dx \le \| n_I\|_{L^2} \| \phi^\beta\|_{L^2}  - \beta \| \phi^\beta \|_{L^2}^2 \le \frac{1}{2\beta} \| n_I\|_{L^2}^2$$
and 
$$\frac 1{\beta} \int_{\{\phi^\beta\le 0\}} e^{\beta \phi^\beta}Ê\beta \phi^\beta dx \le \frac{|\Omega|e^{-1}}{\beta}.$$
This proves that $\cE(\beta) \to 0$ as $\beta \to \infty$. The existence and uniqueness of $\beta$ so that $\cE(\beta) = \cE_1 $ follows from the strict monotonicity of $\cE(\beta)$ in $\beta \in (0,\infty)$. The theorem is proved. 
\end{proof}

\subsection{Arnold's nonlinear stability for fixed ions density}
In this sub-section, we consider the Vlasov-Poisson system for electrons. That is to say $f_-(x,v,t)=f(x,v,t)$ and $\phi$ solve 
\begin{equation}\label{VPB1}  \epsilon \partial_t f + v \cdot \nabla_x f  + \nabla_x \phi \cdot \nabla_v f= \eta_\epsilon Q(f)\end{equation} 
together with the specular boundary condition for $f$, 
coupled with Poisson equation
\begin{equation}\label{Poisson1}
 -\Delta \phi + \la f(x)\ra = n_I (x), \qquad \frac{\partial \phi}{\partial n}_{\vert_{\partial \Omega}} = 0\end{equation}
for fixed ions density $n_I (x)$.
It is worthwhile to study the stability of the steady solution $(F,\Phi)$ given by
\begin{equation}
\label{def-sM}
 F(x,v) = \Big(\frac{\beta }{2\pi }\Big)^{3/2}  e^{- \beta \Big(\frac{ |v |^2}{2} - \Phi (x)\Big )}
\end{equation} 
  and the solution to Poisson equation 
\begin{equation}\label{Poisson2}
  -\Delta \Phi + \la F(x)\ra = n_I (x), \qquad \frac{\partial \Phi}{\partial n}_{\vert_{\partial \Omega}} =0.
\end{equation}

  We study the entropic stability of the stationary solution in the sense of Arnold in his stability theory for two-dimensional Euler flows. We introduce the notion of relative entropy:
$$ \mathcal{H} (f| F): = \iint_{\Omega \times \mathbb{R}^3} \Big[ f \log \Big(\frac{f}{F}\Big) - f + F\Big](x,v)\; dxdv,$$
for measurable functions $f\ge 0$ and $F>0$. One observes that $\mathcal{H}(f|F) = 0$ if and only if $f = F$ almost everywhere.   

\begin{theorem}\label{Arnold Type}\label{theo-Arnold2}  Let $(F, \Phi)$ be any stationary solution given by \eqref{def-sM} and \eqref{Poisson2}, and let $(f,\phi)$ be any smooth solution of the Vlasov-Poisson-Boltzmann system \eqref{VPB1}-\eqref{Poisson1} so that $f$ is rapidly decaying and $\log f$ has polynomial growth in $v$ as $|v|\to \infty$. Then, there holds 
\begin{equation}\label{entropy-ineq}
\begin{aligned}
 \epsilon \frac{d}{dt}\mathcal{H} (f | F)   +    \frac{\beta}{2}\frac{d}{dt} &\int_\Omega |\nabla \phi - \nabla  \Phi|^2  \; dx   = D(f)
   \end{aligned}\end{equation} 
   in which $D(f)$ denotes the entropy dissipation, defined by 
$$    D(f):=  \eta_\epsilon \iint_{\Omega \times \mathbb{R}^3} Q(f) \log f \; dxdv
\le 0.  $$ 
\end{theorem}
   
\begin{proof} Multiplying the Vlasov equation by $\log f$, integrating over $\Omega \times \mathbb{R}^3$, and using the specular boundary condition on $f$, we get 
$$  \epsilon \frac{d}{dt}\mathcal{H} (f)  = D(f) . $$
Hence, by definition,
$$
\begin{aligned}
 \epsilon \frac{d}{dt}\mathcal{H} (f | F)  - D(f) &=  - \iint_{\Omega \times \mathbb{R}^3}  (1+\log F) \partial_t f (x,v,t)\; dxdv
 \\ & =  \iint_{\Omega \times \mathbb{R}^3}  (1+\log F)  \Big[ v \cdot \nabla_x f  + \nabla_x \phi \cdot \nabla_v f - \eta_\epsilon Q(f) \Big]\; dxdv
  \\ & =  \iint_{\Omega \times \mathbb{R}^3}  \Big(\mu  - \beta(\frac{|v|^2}{2} - \Phi)\Big)  \Big[ v \cdot \nabla_x f + \nabla_x \phi \cdot \nabla_v f - \eta_\epsilon Q(f) \Big]\; dxdv,
   \end{aligned}$$
with $\mu = 1+ \frac 32 \log(\frac{\beta}{2\pi})$, in which we have used the explicit form of $F$ as in \eqref{def-sM}. Using the property of $Q(f,f)$, stated in \eqref{Q-property}, the above integral involving $Q(f)$ vanishes.   
Integrating by parts with respect to $x$ and $v$ and using the specular boundary condition on $f$, we get 
\begin{equation}\label{Hff}
\begin{aligned}
 \epsilon \frac{d}{dt}\mathcal{H} (f | F)  -D(f)
 &=  \beta \iint_{\Omega \times \mathbb{R}^3} (\nabla_x \phi - \nabla_x  \Phi) \cdot v f  \; dxdv
 \\
 & = - \beta \int_{\Omega} (\phi -\Phi) \nabla_x \cdot \la v f\ra  \; dx
   =  \beta \int_{\Omega} (\phi -  \Phi) \partial_t \la f \ra  \; dx
   \\
 & =  \beta \int_{\Omega} (\phi -  \Phi) \partial_t \Delta \phi \; dx
    = - \frac \beta 2\frac{d}{dt}\int_{\Omega} |\nabla \phi - \nabla  \Phi|^2\; dx
   \end{aligned}\end{equation}
in which the local conservation of mass was used. This proves the theorem. 
\end{proof}

\begin{remark}\label{rem-weak}{\em The above theorem holds for weak limit of smooth solutions. Precisely, fix $\epsilon>0$, and let $(f_-^n,  \phi^n)$ be any sequence of smooth solutions to the  system, with given initial data $(f_-^0, \phi^0)$ independent of $n$, satisfying
$$ \iint_{\Omega \times \RR^3}  f_-^n \frac{v^2}{2} \; dxdv + \int_\Omega |\nabla \phi^n|^2 \; dx  \le C_0$$
$$ \iint_{\Omega\times \RR^3} f_-^n \log f_-^n \; dxdv \le C_0,\qquad \sup_{(x,v) \in \Omega \times \RR^3} |f_-^n(x,v,t)| \le C_0,$$
for some constant $C_0$, for almost everywhere $t \ge 0$. We assume that $(f_-^n, \phi^n)$ converges weakly to some functions $(f_-, \phi)$ in the following sense: $\nabla \phi^n \rightharpoonup \nabla\phi$ weakly in $L^\infty(\RR_+; L^2 (\Omega))$ and $f_-^n  \rightharpoonup f_-$ weakly in $L^1_{loc}  (\RR_+ \times \Omega \times \RR^3)$. Then, Theorem \ref{theo-Arnold2} holds for $(f_-, \phi)$. 
} {The stability of the ions problem, analyzed below, implies that  $n_I(x,t)$ is slowly varying, and together with the result of Theorem \ref{theo-Arnold2},  this justifies that in many applications, $\beta$ may be taken independent of $t\,.$}
\end{remark}

\section{The reduced ions problem}
As observed above, the Maxwell-Boltzmann approximation reduces the electrons ions problem to a simpler one involving only the ions dynamics. Precisely, 
%
\begin{equation}\label{ions}
\begin{aligned}
&\partial_t f_+ + v \cdot \nabla_x f_+  - \nabla_x \phi \cdot \nabla_v f_+ = 0,\\
& -\Delta \phi + e^{\beta(t)  \phi }=  n_I(x,t)
\\
&
\int_\Omega e^{\beta(t) \phi} \; dx = \int_\Omega n_I(x,t) \; dx = m_0,
\end{aligned}
 \end{equation} 
 for a given positive $m_0$,  in which $\beta(t)$ is determined through the conservation of energy
 \begin{equation}
 \label{ions-beta} \frac{m_0 d}{2\beta(t)} + \frac{1}{2} \int_\Omega |\nabla \phi|^2 \; dx + \iint_{\Omega \times \mathbb{R}^d} \frac{|v|^2}{2} f_+(x,v,t)\; dvdx = \mathcal{E}_0
 \end{equation}
 for some fixed $\mathcal{E}_0> 0$. This is a weakly nonlinear modification of the Vlasov-Poisson system. The classical results there can be adapted to the above reduced ions problem. 
 Here, $\Omega$ is either a bounded open domain or periodic box in $\mathbb{R}^d$. In the former case, we use the specular boundary condition for $f_+$ and the zero Neumann boundary condition for $\phi$. 
 
 Our result in this section is as follows.

 \begin{theorem}[Existence of weak solutions] \label{theo-existence} Assume that the initial data $f_{0,+}\in L^1 \cap L^\infty$ are  compactly supported in $v$ and 
 {that for some fixed $\mathcal E_0$ one has:
 \begin{equation}
 \int_{\Omega\times \mathbb{R}^d} \frac{|v|^2}2f_{0,+}(x,v)dxdv \le a\mathcal E_0 \quad \hbox{with} \, \,\,a<1\,.
 \label{compatibility}
 \end{equation} }
 Then, there is a time $T>0$ so that weak solutions $(f_+, \phi, \beta)$ to the ions problem exist and satisfy 
 $$ f_+ \in L^\infty (0,T; L^1 \cap L^\infty (\Omega \times \mathbb{R}^d)), \qquad n_I \in L^\infty (0,T; L^1 \cap L^\infty (\Omega)),$$
 the electric field $E = -\nabla\phi \in L^\infty(0,T; L^\infty(\Omega))$, and $\beta \in  L^\infty ([0,T])$. Moreover, for $d = 1,2,3$, this solution can be extended globally in time. 
 \end{theorem}

\begin{remark}{\em In the above theorem, the compact support hypothesis on the initial data is assumed for sake of simplicity. One can allow initial data with more general uniform decay, as done in \cite{Lions,Pallard}. 
}\end{remark}

Next, with additional regularity, we have the following uniqueness theorem.

 \begin{theorem}[Uniqueness] \label{theo-uniqueness} Let $T>0$. There exists at most one weak solution $(f_+, \phi, \beta)$
to the reduced ions problem with $v$-compactly supported initial data $f_{0,+}$, provided that
\begin{equation}\label{assmp-U}  \sup_{t\in [0,T]} \sup_{x\in \Omega} \| \nabla_v f_+ \|_{L^2(\mathbb{R}^d)}  +  \int_0^T \| \nabla \phi(s,\cdot)\|_{L^\infty (\Omega)} \; ds  < \infty.\end{equation}
 \end{theorem}

As usual, the proof of existence of solutions, Theorem \ref{theo-existence}, relies on   a-priori estimates. We construct solutions $f_+$ so that 
\begin{equation}\label{EE-f}\iint_{\Omega \times \mathbb{R}^d} \frac{|v|^2}{2} f_+(x,v,t)\; dvdx \le  \mathcal{E}_0,\qquad \iint_{\Omega\times \mathbb{R}^d} f_+(x,v,t) \; dvdx = m_0,\end{equation}
for all $t\ge 0$. It is then straightforward to check that 
$$ \sup_{t \ge 0} \| n_I (\cdot, t)\|_{L^{\frac{d+2}{d}}(\Omega)} \le 2^{\frac {d+2}2}|\mathbb{S}^{d-1}|^{\frac{d+2}d} \|f_+\|^{\frac 2{d+2}}_{L^\infty(  \Omega\times \mathbb{R}^d)}
(\iint_{\Omega\times \mathbb{R}^d}  {|v|^2|} f_+(x,v,t) \; dvdx )^{\frac{d} {d+2}}=C_0.$$

 \subsection{A priori bound on $\beta(t)$}
With  \eqref{ions-beta} we observe that  $\beta(t)$ is bounded below from zero: The fact that $\beta(t)$ also bounded from above follows from the next proposition.
\begin{proposition}\label{prop-beta} For $(\beta, \phi, f_+)$ solution of the ions problem \eqref{ions}-\eqref{ions-beta} the conservation of energy \eqref{ions-beta} is equivalent to the following relation:
\begin{equation}\label{betabound}
\begin{aligned}
 &\beta(t)=e^{\frac{1}{m_0d}( C_0-2 \int_\Omega \beta (t) \phi(x,t)  e^{\beta(t) \phi(x,t)}dx)} \\
&\hbox{ with}\quad C_0=  m_0 d \log \beta (0)+  2\int_\Omega \beta (0)\phi (x,0) e^{\beta (0)\phi(x,0)}\; dx  \,.
\end{aligned}
\end{equation}
\end{proposition}
\begin{corollary}\label{trivial} For $(\beta, \phi, f_+)$ solution of the ions problem $\beta(t)$ is uniformly bounded according to the formula:
\begin{equation}
   \frac{m_0 d}{2\mathcal{E}_0} \le \beta(t) \le e^{\frac{1}{m_0d}( C_0 +2 |\Omega|  e^{-1})}.\label{betabound1}
\end{equation}
\end{corollary} 
\begin{proof}
 The lower bound in the estimate \eqref{betabound1} is a direct consequence of \eqref{ions-beta}, whereas the upper bound follows from \eqref{betabound} with the estimate:
\begin{equation}
- 2\int_{\Omega }  (\beta \phi)  e^{\beta \phi}\; dx \le - 2\int_{\Omega\cap \{ \beta \phi <0\}}  (\beta \phi)  e^{\beta \phi}\; dx + C_0 \le  2e^{-1} |\Omega| + C_0.
\end{equation} 
Given Proposition \ref{prop-beta}, the corollary is proved. \end{proof}

\begin{proof}[Proof of Proposition \ref{prop-beta}] The existence and uniqueness of    $(\beta(t), \phi(t))$ given $f_+(t)$  (in particular for $t=0$)  is proven in Theorem \ref{theo-beta}. To prove \eqref{betabound} we compute 
 $$
\frac{1}{2}\frac{d}{dt} \int_\Omega |\nabla \phi|^2 \; dx 
 =  - \int_\Omega \phi \Delta \phi_t  = \int_\Omega \phi \partial_t n_I - \int_\Omega \phi \partial_t e^{\beta \phi}$$
and 
$$\frac{d}{dt}\iint_{\Omega \times \mathbb{R}^d} \frac{|v|^2}{2} f_+(x,v,t)\; dvdx  =  - \int_\Omega E \cdot n_I u_I  = \int_\Omega \phi \nabla \cdot n_I u_I =  - \int_\Omega \phi \partial_t n_I.$$
This yields 
$$\begin{aligned}
\frac{d}{dt} &\iint_{\Omega \times \mathbb{R}^d} \frac{|v|^2}{2} f_+(x,v,t)\; dvdx  + \frac{1}{2}\frac{d}{dt} \int_\Omega |\nabla \phi|^2 \; dx  
\\&= - \int_\Omega \phi \partial_t e^{\beta \phi} 
= - \frac 1\beta \partial_t \int_\Omega (\beta \phi -1) e^{\beta \phi}\; dx
\\&= - \frac 1\beta \partial_t \int_\Omega \beta \phi  e^{\beta \phi}\; dx
\end{aligned}$$
in which the last equality is due to the conservation of mass. The constraint \eqref{ions-beta} now reads 
\begin{equation}
 - \frac{m_0 d}{2\beta^2} \partial_t \beta - \frac 1\beta \partial_t \int_\Omega \beta \phi  e^{\beta \phi}\; dx =0 \,.\label{prelog}
 \end{equation}
 Or equivalently, 
 \begin{equation}\label{id-beta} m_0 d \log \beta +  \int_\Omega \beta \phi  e^{\beta \phi}\; dx =C_0, \qquad \forall t\ge 0.\end{equation}
and then \eqref{betabound} follows by integration.
\end{proof}


\subsection{Bounds on the electric field}


Let $f_+$ satisfy \eqref{EE-f}. We start with a priori estimates to the following elliptic problem 
\begin{equation}\label{elliptic}
\begin{aligned}
& -\Delta \phi + e^{\beta(t)  \phi }=  n_I(x,t), \qquad \int_\Omega e^{\beta(t) \phi} \; dx =m_0,\\
&\hbox{ with}\quad  \del_n \phi_{|\del \Omega}=0 \quad \hbox{ whenever} \quad \del\Omega \not=\emptyset
\end{aligned}
 \end{equation} 
 with the constraint \eqref{ions-beta}.
For any $p \ge1$, multiplying the elliptic equation by $e^{(p-1) \beta(t) \phi}$, and integrating by parts, we get 
$$ (p-1) \beta(t) \int_\Omega e^{(p-1)\beta(t) \phi} |\nabla \phi|^2 \; dx + \int_\Omega e^{p \beta(t) \phi} \; dx \le \|n_I(\cdot, t) \|_{L^{p}} \| e^{p\beta(t) \phi}\|_{L^1}^{\frac{p-1}{p}}$$
which implies 
\begin{equation}\label{p-H1} 
\|e^{\beta(t) \phi} \|_{L^p} \le  \| n_I(\cdot, t)\|_{L^p}, ,\qquad \forall p \in [1,\infty[ \end{equation}
uniformly in $t\ge 0$ and $p\ge1$.Eventually  by taking $p \to \infty$ in the above inequality, we have also
\begin{equation}\label{bd-ebp}\| e^{\beta(t) \phi (\cdot, t)}\|_{L^p} \le   \| n_I(\cdot, t) \|_{L^p} ,\qquad \forall p \in [1,\infty],\end{equation}
uniformly in $t\ge 0$ and in $\beta(t)$, as long as the right hand side is finite. This yields 
$$ -\Delta \phi = n_I - e^{\beta \phi} \in L^{\frac{d+2}{d}} (\Omega).$$
The standard elliptic problem then yields $\phi \in W^{2,\frac{d+2}{d}}$, whose norm is uniformly bounded in time. In particular, by Sobolev embedding, $\phi$ is uniformly bounded, for $d = 2$ or $3$.



We now write the solution to the elliptic problem as 
\begin{equation}\label{phi-K} \phi = \int_\Omega K(x,y) \Big[ n_I (y,t) - e^{\beta(t) \phi (y,t)} \Big] \; dy \end{equation}
in which $K(x,y)$ denotes the Green kernel of the Laplacian on $\Omega$ with the Neumann boundary condition or periodic boundary condition. 
It is classical that $$ |\partial_x^kK(x,y)| \le C_0 |x-y|^{2-d-k}, \qquad k\ge 0$$
for $d \ge 3$. For $d =2$, $K(x,y)$ is of order of $\log |x-y|$. 

\begin{lemma}\label{lem-bdE} With $E = -\nabla \phi$, there hold
$$\| E(\cdot, t)\|_{L^\infty}\le C_0 \|n_I(\cdot, t)\|_{L^1}^{\frac{1}{d}} \|n_I(\cdot, t)\|_{L^\infty}^{\frac{d-1}{d}}$$
uniformly in $t \ge 0$. 
\end{lemma} 
\begin{proof} The proof is straightforward, using \eqref{phi-K} and \eqref{bd-ebp}.\end{proof}

\subsection{A priori bounds on ions density}

Given the field $E(x,t)$, starting from $(x,v) \in \Omega \times \mathbb{R}^d$, the particle trajectories $(X(t), V(t))$ are defined by the ODEs
$$ \dot X = V, \qquad \dot V = E(X(t),t) $$
as long as $X(t) $ remains in the interior of $\Omega$. In the case $\Omega$ has a boundary, we let $t_0$ be the positive time when $X(t_0)$ hits the boundary, that is $X(t_0) \in \partial \Omega$. The trajectory is then continued by the ODE dynamics, with the new ``initial'' condition:
$$ X(t_0) = \lim_{t \to t_0^-} X(t), \qquad V(t_0): = \lim_{t \to t_0^-} \Big[ V(t) - 2(V(t)\cdot n(X(t)))n(X(t)) \Big] ,$$
which of course correspond to the specular boundary condition of particles, and so on, in case of multiple reflections. The backward trajectory $(X(t), V(t))$ is defined in the similar way, {for $0<t<t_0\,.$}

Then, the solution $f_+$ to the Vlasov equation is constructed through 
\begin{equation}\label{def-fion}f_+(x,v,t) = f_{0,+} (X(-t), V(-t)), \qquad \forall t\ge 0, \quad \forall (x,v)\in \Omega \times \mathbb{R}^d,\end{equation}
with $(X(0),V(0)) = (x,v)$. 
With $f_{0,+} (x,v) = 0$ for all $|v| \ge K_0$ for some positive $K_0$, we first compute the growth of the support in $v$. By definition, as long as $X(t) \in \Omega$, there holds 
$$ \frac{d}{dt}|V|^2 = 2 E \cdot V.$$
When $X(t)$ meets $\partial \Omega$, $|V(t)|$ is conserved under the specular reflection. Hence, for all $(x,v)\in \Omega \times \mathbb{R}^d$ with $|v|\le K_0$, we have 
\begin{equation}\label{bd-Vt} |V(t)| \le |v|  + \int_0^{|t|} \| E\|_{L^\infty} \; ds 
\end{equation}
for all $t \in \mathbb{R}$. Now, using the characteristic equation \eqref{def-fion}, we have  
$$ | n_I(x,t)| \le \int_{\mathbb{R}^d} | f_{0,+} (X(-t), V(-t)) | \; dv \le C_0 (K_0 + |V(-t)|)^{d}.$$
 
 Combining the last two estimates, we have obtained
\begin{equation}\label{bd-rho} \| n_I(\cdot, t) \|_{L^\infty} \le C_0 + C_0 \Big( \int_0^t \| E(\cdot, s)\|_{L^\infty}\; ds \Big)^d.
\end{equation}
Together with Lemma \ref{lem-bdE} and the fact that $n_I(\cdot , t) \in L^1$, the above yields
$$
\| n_I(\cdot, t) \|_{L^\infty} \le C_0 + C_0 \Big( \int_0^t  \|n_I(\cdot, s)\|_{L^\infty}^{\frac{d-1}{d}}\; ds \Big)^d.
$$
Hence, the Gronwall's inequality gives 
\begin{equation}\label{bd-rho1} 
 \| n_I(\cdot, t) \|_{L^\infty} \le C_T
 \end{equation}
 for all $t \in [0,T]$, for some positive $T$. In the two dimensional case, $T = \infty$.

\subsection{Averaging lemma}
In the sequel, we also need a priori compactness on the average of $f_+$ which follows from the classical $L^2$ averaging lemma (\cite{Lions}). Indeed, we write the Vlasov equation as 
$$ \partial_t f_+ + v \cdot \nabla_x f_+ = -\nabla_v (E f_+) .$$
Here, from the apriori estimates, $E \in L^\infty$ and $f \in L^1 \cap L^\infty$. Hence, 
$$\| f_+\|^2_{L^2 (0,T; L^2 (\Omega \times \RR^3))}  \le \| f_+\|_{L^\infty} \| f_+\|_{L^1 (0,T; L^1 (\Omega \times \RR^3))}  \le \| f_+\|_{L^\infty} \|n_I\|_{L^1((0,T) \times \Omega)}$$
and $$ \| E f_+ \|_{L^2 (0,T; L^2 (\Omega \times \RR^3))}  \le \| E\|_{L^\infty} \| f_+\|_{L^2 (0,T; L^2 (\Omega \times \RR^3))} .$$
By the classical averaging lemma and the fact that $f_+(x,v,t)$ is compactly supported, we have
$$ \int_{\RR^3} f_+(x,v,t) \varphi(v) \; dv \in H^{1/4} ((0,T)\times \Omega)$$
together with the uniform bound 
$$ \Big\| \int_{\RR^3} f_+(\cdot, \cdot,v) \varphi(v) \; dv  \Big\|_{H^{1/4} ((0,T)\times \Omega)} \le C_\varphi  \| E\|_{L^\infty} \| f_+\|_{L^2 (0,T; L^2 (\Omega \times \RR^3))} $$
for any test function $\varphi(v)$ in $C^\infty (\RR^3)$   
{ and in particular  for $\phi(v)=1 $ or $\phi(v)=\frac{|v|^2}2\,,$ used below.}


\subsection{Proof of local well-posedness}
The existence of local solutions to the ions problem \eqref{ions} now follows {with minor modifications}  the standard iteration procedure. Indeed, we construct $(\beta_n, \phi_n, f_n)$ as follows. 
Let $f_{0,+}\in (L^\infty \cap L^1)(\Omega \times \mathbb{R}^d) $ be any initial data  compactly supported in $v$ and satisfying :
$$\iint_{\Omega \times \mathbb{R}^d} \frac{|v|^2}{2} f_{0,+}(x,v)\; dvdx \le a  \mathcal{E}_0,\qquad \iint_{\Omega\times \mathbb{R}^d} f_{0,+}(x,v) \; dvdx = m_0$$
for some $a<1$ { (cf. \eqref{compatibility})}.
Set $f_0 (x,v,t) = f_{0,+}(x,v)$. We start the iteration with $n =0$. 
We denote in the sequel $\rho_n  (x,t) = \la   f_n (x,.,t) \ra .$

\begin{itemize}

\item  We will construct the unique solution $(\beta_n, \phi_n)$ to the elliptic problem 
\begin{equation}
\begin{aligned}
 -\Delta \phi_n + e^{\beta_n \phi_n} &= \rho_n , 
\qquad \int_\Omega e^{\beta_n\phi_n} \; dx  =m_0,
\\
 \frac{m_0 d}{2\beta_n} + \frac{1}{2} \int_\Omega |\nabla \phi_n|^2 \; dx & = \mathcal{E}_0 - \iint_{\Omega \times \mathbb{R}^d} \frac{|v|^2}{2} f_n(x,v,t)\; dvdx. \label{problem}
 \end{aligned}
\end{equation}
%
 
\item Then we will  construct $f_{n+1}$ by solving the linearized Vlasov equation 
\begin{equation}
 \partial_t f_{n+1} + v \cdot \nabla_x f_{n+1} - \nabla_x \phi_n \cdot \nabla_v f_{n+1} =0\label{liouvillen}
 \end{equation}
with the same initial data $f_{n+1}(x,v,0) = f_{0,+}(x,v)$. 

\end{itemize}


{However to solve the elliptic problem \eqref{problem} one needs to ensure that the quantity
\begin{equation}
 \overline{E_n(t)} =\mathcal{E}_0 - \iint_{\Omega \times \mathbb{R}^d} \frac{|v|^2}{2} f_n(x,v,t)\; dvdx. 
\end{equation} 
remains strictly positive. For a genuine solution this follows obviously from the energy conservation \eqref{ions-beta} and on the uniform bound \eqref{betabound1}, but for a iterative solution, this requires some extra argument. 
By iteration a sequence of decreasing positive times $0<T_n$ is introduced. They are characterized by the fact that  $ \overline{E_n(t)}$ is strictly positive for $0<t<T_n\,.$ Hence on such interval the solution of  \eqref{problem} is well defined. On any such interval, bounds for $(f_n,\phi_n,\beta_n)$ are derived uniformly in $n$. Hence, it is shown (cf. Lemma \ref{positif}) that
\begin{equation}
T_-=\inf T_n
\end{equation} 
is a strictly positive number which depends  only on the properties of the data at $t=0\,.$}

For the $n$-uniform bound, applying Lemma \ref{lem-bdE} and the bound \eqref{bd-rho} to the above iterative scheme, we obtain 
\begin{equation}
\begin{aligned}
 \| \rho_{n+1}(\cdot, t) \|_{L^\infty} 
 &\le  C_0 + C_0 \Big( \int_0^t \| E_n(\cdot, s)\|_{L^\infty}\; ds \Big)^d
 \\
 &\le C_0 + C_0 \Big( \int_0^t \| \rho_n(\cdot, s)\|_{L^\infty}^{\frac{d-1}{d}} \; ds \Big)^d
\end{aligned}
\end{equation}
for all $n\ge 0$. By iteration and the previous estimates, this proves that 
\begin{equation}\label{uni-bd} \| \rho_n(\cdot, t) \|_{L^\infty} \le C(t), \qquad \| E_n(\cdot, t) \|_{L^\infty}\le C(t),\qquad |\beta_n(t)|\le C(t),
\end{equation}
uniformly in $n$, for all positive time $t$ ($d=1,2$), and for $t \in [0,T]$ for some positive time $T$ ($d\ge 3$). Here, $C(t)$ denotes some continuous function in $t$.

{Eventually with $C_T=\sup_{0<t<T} C(t)$,  the above estimates can be used to prove the following.
\begin{lemma}\label{positif} 

1.  For any $f_{n+1}(x,v,t)$ one has, for $0<t<T$, the estimate:
\begin{equation}
\int_\Omega  \la \frac{|v|^2}2 f_{n+1}(t)\ra dx\le (2C^{\frac32}_T t+ ( \int_\Omega \la \frac{|v|^2}2 f_{n+1}(0)\ra dx)^{\frac12})^2\label{positive2}
\end{equation}
2.   As long as $t$ is small enough to satisfy the relation
\begin{equation}
 (2C^{\frac32}_T t+ (a{\mathcal {E}}_0)^{\frac12})^2< {\mathcal {E}}_0
\end{equation}
in which $a>1$ is given by \eqref{compatibility}, the expression:
$$
{\mathcal {E}}_0-\int_\Omega  \la \frac{|v|^2}2 f_{n+1}(t)\ra dx
$$
remains strictly positive.
\end{lemma} 
\begin{proof}
From the equation \eqref{liouvillen}, one deduces the following usual relation:
\begin{equation}
\frac{d}{dt}\int_\Omega  \la \frac{|v|^2}2 f_{n+1}\ra dx=\int_\Omega \nabla_x \phi_n \cdot \la v f_{n+1}\ra dx\end{equation}
Therefore, together with  the Cauchy-Schwarz's inequality and \eqref {uni-bd}, one has the following estimate:
\begin{equation}
\begin{aligned}
&\frac{d}{dt}\int_\Omega  \la \frac{|v|^2}2 f_{n+1}\ra dx \le \int_{\Omega\times\mathbb{R}^d}|\nabla_x \phi_n(x)| | v f_{n+1}|dxdv\\
&\le   (\int_{\Omega\times\mathbb{R}^d}|\nabla_x \phi_n(x)|^2  f_{n+1} dxdv)^{\frac12}( \int_{\Omega} \la  \frac{| v |^2}2f_{n+1}\ra dx )^\frac12\\
&\le (C(t) \int_{\Omega} |\nabla_x \phi_n(x)|^2dx)^{\frac12}( \int_{\Omega} \la  \frac{| v |^2}2f_{n+1}\ra dx )^\frac12 
\\&\le  C_T( \int_{\Omega} \la  \frac{| v |^2}2f_{n+1}\ra dx )^\frac12\, \hbox{ for } t\in [0,T]\,.
\end{aligned}
\end{equation}
Hence, \eqref{positive2} follows by integration. The second statement is a direct consequence of the first. It is important to observe that the estimates involve only the quantity $C_T$, which has been globally evaluated. 
\end{proof}}
 Now we can consider the convergence  of the sequence $(f_n, \phi_n,\beta_n)$.
 Up to a subtraction of subsequences, $f_n \rightharpoonup f$ in $L^p(\Omega \times \mathbb{R}^d)$, $E_n \rightharpoonup E$ in $L^p(\Omega)$, and $\beta_n(t) \to \beta(t)$ for almost every where $t\in [0,T]$. By view of the elliptic problem for $\phi_n$, we in fact have $E_n = - \nabla \phi_n \in L^\infty (0,T; W^{1,p}(\Omega))$ for all $p\ge 1$. 

To gain regularity in time, we use the averaging lemma, yielding 
\begin{equation}\label{compactfn}\Big\| \int_{\RR^3} f_n(\cdot, \cdot,v) \varphi(v) \; dv  \Big\|_{H^{1/4} ((0,T)\times \Omega)} \le C_\varphi  \| E_n\|_{L^\infty} \| f_n\|_{L^2 (0,T; L^2 (\Omega \times \RR^3))} \le C_T C_\varphi\end{equation}
for any test function $\varphi(v)$ in $C^\infty (\RR^3)$.  Now we can pass to the limit of $n \to \infty$. We fix a test function of the form $\theta(x,t) \varphi(v)$. We get 
$$ \begin{aligned}\int_0^T &\iint_{\Omega \times \RR^3} \nabla_v \cdot ((\nabla_x \phi_n) f_{n+1}(x,v,t)) \theta(x,t) \varphi(v) \; dxdv dt 
\\& =
 -\int_0^T \int_\Omega \nabla_x \phi_n \theta(x,t) \cdot \Big( \int_{\RR^3} f_{n+1} (x,v,t) \nabla_v \varphi \; dv \Big) \; dxdt 
\\&\rightarrow 
 -\int_0^T \int_\Omega \nabla_x \phi \theta(x,t) \cdot \Big( \int_{\RR^3} f (x,v,t) \nabla_v \varphi \; dv \Big) \; dxdt 
 \end{aligned}$$ 
as $n \to \infty$. Similarly for the transport operator $\partial_t f_n + v \cdot \nabla_x f_n$, we obtain 
$$ \partial_t f + v \cdot \nabla_x f - \nabla_x \phi \cdot \nabla_v f = 0$$
in the weak sense. Now, we consider the elliptic problem 
$$\begin{aligned} - \Delta \phi_n + e^{\beta_n \phi_n} &= \rho_n(t) \\
 \frac{m_0 d}{2\beta_n} + \frac{1}{2} \int_\Omega |\nabla \phi_n|^2 \; dx  &= \mathcal{E}_n(t): = \mathcal{E}_0 -  \iint_{\Omega \times \mathbb{R}^d} \frac{|v|^2}{2} f_n(x,v,t)\; dvdx .\end{aligned}$$
Since $f_n$ is compactly supported in $v$, the compactness property \eqref{compactfn} in time for $f_n$ yields the compactness for $\rho_n$ and $\mathcal{E}_n$. The above elliptic problem has data $\rho_n(t)$ and $\mathcal{E}_n(t)$ converges pointwise in time to $\rho(t)$ and $\mathcal{E}(t)$, for almost every time $t \in [0,T].$ Now, for each fixed time $t$, $\beta_n$ and $\phi_n$ are bounded in $\RR$ and $W^{2,p}(\Omega)$, and so, up to a subtraction of subsequences, they converge strongly to $\beta(t)$ and $\phi(x,t)$ in $\RR$ and $H^1(\Omega)$, respectively.  In addition, for each time $t$, $(\beta(t), \phi(x,t))$ solves 
$$\begin{aligned} - \Delta \phi + e^{\beta \phi} &= \rho(t) \\
 \frac{m_0 d}{2\beta} + \frac{1}{2} \int_\Omega |\nabla \phi|^2 \; dx  &= \mathcal{E}(t) = \mathcal{E}_0 -  \iint_{\Omega \times \mathbb{R}^d} \frac{|v|^2}{2} f(x,v,t)\; dvdx .\end{aligned}$$
Now by uniqueness of the above elliptic problem, $(\beta, \phi)$ is thus a solution to the reduced ions problem. This yields a local solution. 
\begin{remark}{\em  The use of the averaging lemma in the present proof seems to be an ``overkill", since usually time regularity in a ``weak space" is deduced from the equations and the Aubin-Lions theorem can be used. However in the present case the time regularity is obtained for $\rho_n(t)$ and $\la \frac{|v|^2}2 f_n(x,v,t)\ra$, which is sufficient for the almost everywhere point wise convergence of $(\beta_n (t),\phi_n(x,t))$. Since the mapping $(\rho_n(t),\la \frac{|v|^2}2 f_n(x,v,t)\ra)\rightarrow (\beta_n (t),\phi_n(x,t)) $ is non linear and not explicit, the use of the above averaging lemma 
to obtain the almost everywhere convergence seems to be the simpler approach.}
\end{remark}

 \subsection{Proof of global well-posedness} 
In the two dimensional case, the linear Gronwall inequality yields at once the uniform bound \eqref{uni-bd} for all time $t$. Hence, the previous analysis provides a global solution to the reduced ions problem.

It remains to consider the three-dimensional case. By a view of \eqref{bd-Vt} and \eqref{bd-rho}, it suffices to prove 
\begin{equation}\label{bd-ESch} \int_0^t \| E(X(s),s)\|_{L^\infty} \; ds \le C_0 |V(t)|^\alpha + C_0\end{equation}
for some positive constant $\alpha<1$. 
The boundedness of $V(t)$ and hence $\rho(t)$ then follows. We follow the proof of Schaeffer for the classical 3D Vlasov-Poisson system. Indeed, let us write the Poisson equation as  
$$-\Delta \phi = n_I - e^{\beta \phi}$$
 and hence, 
$$ \begin{aligned}
E(x,t) &=  - \int_\Omega \nabla_x K(x,y) \Big[ n_I (y,t) - e^{\beta(t) \phi (y,t)} \Big] \; dy
\\
&=: E_1(x,t) + E_2(x,t).
\end{aligned}$$ 
Since $e^{\beta(t) \phi (y,t)}$ is bounded, $E_2(t)$ is uniformly bounded. The bound \eqref{bd-ESch} for $E_1(x,t)$ follows identically from the proof of Schaeffer for the classical Vlasov-Poisson system, using the boundedness of $f_+$ and of the total kinetic energy of $f_+$; see, for instance, \cite{Sch91,Glassey}. This completes the proof of Theorem \ref{theo-existence}. 
 

\subsection{Proof of uniqueness}

 
In this section, we prove the uniqueness of solutions of the ion problem. Indeed, let $(\beta_1, \phi_1, f_1)$  and $(\beta_2, \phi_2, f_2)$  be the two solutions to \eqref{ions} and \eqref{ions-beta}, with the same compactly supported initial data $f_0$. We assume that 
\begin{equation}\label{assmp-bdE} \int_0^t \Big( \| E_{1}(s,\cdot)\|_{L^\infty} +  \| E_{2}(s,\cdot)\|_{L^\infty}\Big) \; ds  < \infty\end{equation}
and 
 \begin{equation}\label{id-energy} \frac{m_0 d}{2\beta_j(t)} + \frac{1}{2} \int_\Omega |\nabla \phi_j|^2 \; dx + \iint_{\Omega \times \mathbb{R}^d} \frac{|v|^2}{2} f_j(x,v,t)\; dvdx = \mathcal{E}_0\end{equation}
 for $j = 1,2$, and for the same energy constant $E_0$. We also assume that 
 $$ 
 \sup_{x,t} \| \nabla_v f_1 \|_{L^2(\mathbb{R}^d)} 
 <\infty.$$ 
In the end of this section, we shall verify the above assumptions when $\Omega = \mathbb{T}^d$. 

 
 
 We show that 
 $$ \beta_1 = \beta_2, \qquad \phi_1 = \phi_1, \qquad f_1 = f_2.$$
 From the identity \eqref{id-beta}, $\beta_j(t)$ remains bounded. As a consequence of \eqref{bd-Vt} and \eqref{assmp-bdE}, the velocity support of $f_j(x,v,t)$ is bounded, for $j=1,2$.  
For convenience, let us denote 
$$ \beta = \beta_1 - \beta_2, \qquad \phi = \phi_1 - \phi_2, \qquad f = f_1 - f_2,$$
and set 
$$ y(t) = \iint_{\Omega\times \mathbb{R}^d} |f(x,v,t)|^2 \; dxdv .$$

The uniqueness follows directly from the following proposition. 
\begin{proposition} There holds 
$$ \frac{d}{dt}y(t) \le C_0 \Big( y(t) + y(t)^2 \Big).$$
\end{proposition}

\begin{proof} First, the difference $f = f_1 - f_2$ solves the following Vlasov equation
$$ \partial_t f + v \cdot \nabla_x f - \nabla_x(\phi_1 + \phi)\cdot \nabla_v f = \nabla_x \phi \cdot \nabla_v f_1.$$
By assumption that $\sup_{x,t} \| \nabla_v f_1 \|_{L^2(\mathbb{R}^d)} <\infty$, the standard energy estimate yields 
\begin{equation}\label{y-f} \frac12 \frac{d}{dt} \| f\|_{L^2}^2 \le C_0 \Big( \| f\|_{L^2}^2 +\|\nabla \phi\|_{L^2}^2\Big) ,\end{equation}
for some universal constant $C_0$ that depends on $ \sup_{x,t} \| \nabla_v f_1 \|_{L^2(\mathbb{R}^d)}$. 

Next, we use the Poisson equation for $\phi$, which now reads
\begin{equation}
\label{diff-p}
-\Delta \phi + e^{\beta_1 \phi_1} - e^{\beta_2 \phi_2} = \rho = \int_{\mathbb{R}^d} f(x,v,t) \; dv.
\end{equation}
We write 
$$e^{\beta_1 \phi_1} - e^{\beta_2 \phi_2}  = e^{\beta_1 \phi_1} - e^{\beta_1 \phi_2}  + e^{\beta_1 \phi_2} - e^{\beta_2 \phi_2}  $$
and use the fact that $|x-y|^{p-2}(e^x - e^y)(x-y) \ge \theta_0 |x-y|^p$, for all $x,y$ in a compact set and all $p>1$. Noting that $\beta_j, \phi_j$ are uniformly bounded and multiplying the elliptic equation by $|\phi|^{p-2} \phi$, we easily obtain 
\begin{equation}\label{Lp-phi} \| \phi\|_{L^p} \le C_0 \Big( |\beta| + \| \rho \|_{L^p}\Big), \qquad \forall ~p>1.\end{equation}


To obtain a better estimate, we write 
$$ 
\begin{aligned}
e^{\beta_1 \phi_1} - e^{\beta_2 \phi_2} 
&= e^{\beta_1 \phi_1} \Big( 1 - e^{\beta_2\phi_2 - \beta_1 \phi_1}\Big) 
\\
&= e^{\beta_1 \phi_1} \Big(\beta_1\phi_1 - \beta_2 \phi_2 + R_{\beta, \phi}\Big)
\end{aligned}$$
in which $R_{\beta, \phi} = \mathcal{O} (|\beta_1 - \beta_2|^2 + |\phi_1 - \phi_2|^2)$. We further write 
$$ 
\begin{aligned}
e^{\beta_1 \phi_1} - e^{\beta_2 \phi_2} 
&= \frac12 e^{\beta_1 \phi_1} \Big((\beta_1+ \beta_2)(\phi_1  - \phi_2) + (\beta_1- \beta_2) (\phi_1 + \phi_2) + 2 R_{\beta, \phi}\Big)
\end{aligned}$$
We next multiply the elliptic equation \eqref{diff-p} by $-2e^{-\beta_1 \phi_1} \Delta \phi$, upon using the above identity and recalling that $\phi = \phi_1 - \phi_2$ and $\beta = \beta_1 - \beta_2$, we obtain  
 $$ \int_\Omega \Big[ 2 e^{-\beta_1 \phi_1} |\Delta \phi|^2 + (\beta_1 + \beta_2) |\nabla \phi|^2  - \beta  (\phi_1 + \phi_2) \Delta \phi  -  R_{\beta, \phi} \Delta \phi\Big] = - 2\int_\Omega \rho e^{-\beta_1 \phi_1} \Delta \phi .$$
Together with the Young's inequality, this yields 
\begin{equation}\label{phi12-bd} 
\begin{aligned}\int_\Omega 
&\Big[ e^{-\beta_1 \phi_1} |\Delta \phi|^2 + (\beta_1 + \beta_2) |\nabla \phi|^2 - \beta  (\phi_1 + \phi_2) \Delta \phi \Big] 
\\&\le C_0 
\Big( |\beta|^4  + \int_\Omega ( |\phi|^4  + |\rho|^2)\Big)
\end{aligned}\end{equation}
in which the bound on remainder $R_{\beta, \phi} = \mathcal{O} (|\beta|^2 + |\phi|^2)$ was used.

We now use the fact that the energy for the two solutions are the same; see \eqref{id-energy}. Subtracting one to another, we get  the conservation of the energy
 $$ \frac{m_0 d (\beta_2 - \beta_1)}{2\beta_1 \beta_2} + \frac{1}{2} \int_\Omega ( |\nabla \phi_1|^2 - |\nabla\phi_2|^2) \; dx + \iint_{\Omega \times \mathbb{R}^d} \frac{|v|^2}{2} (f_1 - f_2)\; dvdx =0.$$  
Recalling $\phi = \phi_1 - \phi_2$ and $\beta = \beta_1 - \beta_2$, we multiply the above by $-2\beta$
and note that the middle term can be written as 
$$\frac12 \int_\Omega ( |\nabla \phi_1|^2 - |\nabla\phi_2|^2)  = - \frac12\int_\Omega (\phi_1 + \phi_2) \Delta (\phi_1 - \phi_2) = -\frac12 \int_\Omega (\phi_1 + \phi_2) \Delta \phi.$$
We get 
\begin{equation}\label{beta12-bd} \frac{2m_0 d \beta^2 }{2\beta_1 \beta_2}  +   \beta \int_\Omega  (\phi_1 + \phi_2) \Delta \phi  = 2\beta  \iint_{\Omega \times \mathbb{R}^d} \frac{|v|^2}{2} f \; dvdx.\end{equation} 
Here in \eqref{beta12-bd} we note that the kinetic energy is bounded by $\| f\|_{L^2}$, since $f$ is compactly supported in $v$. Adding \eqref{phi12-bd} and \eqref{beta12-bd} together and recalling that $\beta_j$ are bounded below away from zero, we obtain at once 
\begin{equation}\label{bd-pbeta} 
\begin{aligned}
|\beta|^2 + \|\nabla \phi\|^2_{L^2}+ \|\Delta \phi \|^2_{L^2} \le C_0 \Big( |\beta|^4  +  \|\phi\|_{L^4}^4  +\|f\|_{L^2}^2 \Big) .
\end{aligned}\end{equation}
Now using the $L^p$ bound \eqref{Lp-phi}, with $p=4$, and recalling that $f$ is compactly supported, we obtain from the previous estimate 
\begin{equation}\label{bd-pbeta1} 
\begin{aligned}
|\beta|^2 + \|\nabla \phi\|^2_{L^2}+ \|\Delta \phi \|^2_{L^2} \le C_0 \Big( |\beta|^4   +\|f\|^2_{L^2} + \|f\|_{L^2}^4 \Big) .
\end{aligned}\end{equation}

It remains to take care of $|\beta|^4$ on the right-hand side. To this end, we shall prove that $\beta_j(t)$ is continuous in time. It suffices to show the continuity of $\beta_1$. Indeed, we note that $f_1$ is continuous in time, since $f_1$ is a $C^1$ function with respect to $x,v$, and
$$\partial_t f_1 = - v \cdot \nabla_x f_1 + \nabla_x \phi_1 \cdot \nabla_v f_1.$$
Now we fix $f_1$, and study the elliptic problem 
$$ -\Delta \phi_1 + e^{\beta_1 \phi_1} = \rho_1(t) , \qquad E(\beta_1)= E_0(t): = E_0 - \iint_{\Omega \times \mathbb{R}^d} \frac{|v|^2}{2} f_1(x,v,t) \; dxdv$$
in which $ E(\beta_1): = \frac{m_0 d}{\beta_1} + \frac12 \int_\Omega |\nabla \phi_1 |^2 $. Here, $\rho_1(t)$ and $E_0(t)$ are two continuous functions. Fix a $t$ and let $t_n$ be a sequence so that $t_n \to t$. Then, there are unique solutions $(\beta_1(t_n), \phi_1(t_n))$ and $(\beta_1(t), \phi_1(t))$ to the elliptic problems, corresponding to $(\rho(t_n), E_0(t_n))$ and $(\rho(t), E_0(t))$, respectively. In addition, we have $\beta_1(t_n)$ and $\phi_1(t_n)$ are uniformly bounded in $\mathbb{R}$ and $H^2$, respectively. Hence, there is a subsequence $t_{n_k}$ so that $(\beta_1(t_{n_k}), \phi_1(t_{n_k}))$ converges, and by uniqueness, the whole series converges to the same limit $(\beta_1(t), \phi_1(t))$. In particular, this yields the continuity of $\beta_1(t)$. 

Finally, by the continuity, the term $|\beta|^4$ on the right-hand side of \eqref{bd-pbeta1} can be absorbed into the left-hand side, for small $t$, since $\beta(0) = 0$, yielding 
$$\begin{aligned}
|\beta|^2 + \|\nabla \phi\|^2_{L^2}+ \|\Delta \phi \|^2_{L^2} \le C_0 \Big( \|f\|^2_{L^2} + \|f\|_{L^2}^4 \Big) .
\end{aligned}
$$
Putting this into \eqref{y-f} finishes the proof of the proposition, and hence the proof of the uniqueness of the solutions to the ion problem \eqref{ions}-\eqref{ions-beta}. 
\end{proof}

We end the section by proving the following propagation of regularity in the torus $\Omega = \mathbb{T}^d$. For uniqueness, it suffices to prove the propagation of regularity, assumed in Theorem \ref{theo-uniqueness}, in a short time interval. 

\begin{proposition} Let $\Omega = \mathbb{T}^d$ and $(\beta, \phi, f_+)$  be a solution to \eqref{ions} and \eqref{ions-beta} with compactly $v$-supported and bounded  initial data $f_{+,0}$. If we assume that 
$$\| \nabla_x f _{+,0}\|_{L^\infty_{x,v}} +  \| \nabla_v f _{+,0}\|_{L^\infty_{x,v}} < \infty,$$ then for small positive time $T$, there holds 
$$ \sup_{t\in [0,T]} \Big( \| \nabla_x f _+(t)\|_{L^\infty_{x,v}} +  \| \nabla_v f _+(t)\|_{L^\infty_{x,v}} \Big) < \infty.$$
\end{proposition}
\begin{proof} The proof is straightforward. Indeed, $\nabla_x f_+$ and $\nabla_v f_+$ satisfy 
$$\begin{aligned}
 \Big( \partial_t + v \cdot \nabla_x - E \cdot \nabla_v \Big) \nabla_x f_+ &= \sum_k \nabla_x E_k \partial_{v_k} f_+  
\\
 \Big( \partial_t + v \cdot \nabla_x - E \cdot \nabla_v \Big) \nabla_v f_+ &=  -\nabla_x f_+.  
 \end{aligned}$$
This yields 
$$ \| \nabla_v f_+ (t)\|_{L^\infty} \le \int_0^t \| \nabla_x f_+ (s)\|_{L^\infty}  \; ds $$
and 
$$\| \nabla_x f_+(t)\|_{L^\infty} \le \int_0^t \| D_x^2 \phi\|_{L^\infty} \| \nabla_v f_+ (s)\|_{L^\infty}  \; ds .$$
Here, $\phi $ solves the elliptic problem $-\Delta\phi = n_I - e^{\beta \phi}$ and hence 
$$ -\Delta D_x \phi = D_x n_I - D_x e^{\beta \phi} .$$
 Hence, applying Lemma \ref{lem-bdE}, for $D_x\phi$, together with the fact that $\Omega$ is bounded, yields at once 
  $$
  \begin{aligned}
   \| D_x^2 \phi\|_{L^\infty} &\le C_0 \|D_x n_I\|_{L^\infty} + C_0 \| D_x e^{\beta \phi}\|_{L^\infty}
   \\
   &\le C_0 \|D_xf_+\|_{L^\infty} + C_0 \| e^{\beta \phi}\|_{L^\infty}\| D_x \phi \|_{L^\infty}
   \end{aligned}$$
in which we noted that $f_+$ is compactly supported in $v$. Recall that $\|e^{\beta \phi}\|_{L^\infty }\le \|n_I\|_{L^\infty} \le C_0$ and $\| D_x \phi \|_{L^\infty} \le C_0 \|n_I\|_{L^\infty} \le C_1$, since $f_+\in L^\infty$. Hence, 
$$\| \nabla_x f_+(t)\|_{L^\infty} \le C_0 \int_0^t (1 +\|  \nabla_x f_+(s)\|_{L^\infty})\| \nabla_v f_+ (s)\|_{L^\infty}  \; ds .$$ 
The proposition follows at once from the standard nonlinear Gronwall's lemma. 
\end{proof}
\section{Conclusion}

We end the paper with some remarks:
\begin{itemize}

\item For the interaction for the evolution of a plasma involving ions and electrons an approximation of the density of electrons is often used and it is referred as the Maxwell-Boltzmann relation. The aim of the present contribution was to fully justify  this approach assuming a kinetic description for the electrons where the characteristic interaction time is faster than rate of relaxation to equilibrium.
This seems the most natural way to obtain a proof. On the other hand, as indicated by the point ii) of Theorem \ref{theo-formal}, considering a macroscopic equation for the ions  seems compatible with the present approach. And eventually one should observe that in some case the counterpart of the Maxwell-Boltzmann relation can be derived for some well adapted macroscopic description; cf. \cite{JP,Grenier} for an example and references.

\item One may wonder at getting a electrons temperature which is constant with respect to the space variable. But recall we deal here with a modelling at the scale of the Debye length (for instance some tens or hundreds of Debye lengths) and at this scale it is natural that the electrons temperature is constant even if it is not the case at a much larger scale.\item The main difficulty towards a complete  proof that is valid in full generality seems to come  from the fact that the conservation of energy for large time for the solution of the Boltzmann equation, even formally true and expected in general at the level of mathematical rigor,  remains an open problem. 
This difficulty  persists in the presence of a electromagnetic interaction. This is the reason why some uniform regularity hypothesis is assumed in the theorem, Theorem \ref{theo-formal}. 


\item In the present contribution the coupling between the ions and electrons is described through the effect of the electric field, magnetic effect and collisions between ions and electrons are ignored, such issues may be the object of future works.

\end{itemize}
  ~\\
{\em Acknowledgement.} The authors thank Claudia Negulescu for introducing us the problem. We also thank CMLS (Polytechnique), ICES (Austin), Penn State University, and Wolfgang Pauli Institute (Vienna) for support and hospitality of a visit during which part of this work was carried out.

\bibliographystyle{plain}

\end{document}